\documentclass[11pt,leqno]{amsart}
\usepackage{epsfig}
\usepackage{amssymb}
\usepackage{amscd}
\usepackage[matrix,arrow]{xy}

\newtheorem{theo}{Theorem}[section]
\newtheorem{lemma}[theo]{Lemma}
\newtheorem{defi}[theo]{Definition}
\newtheorem{prop}[theo]{Proposition}

\newtheorem{cor}[theo]{Corollary}
\newtheorem{remark}[theo]{Remark}

\newtheorem{example}[theo]{Example}
\numberwithin{equation}{section}

\def\Qcoh{\operatorname{Qcoh}}
\def\lto{\longrightarrow}

\def\bR{{\mathbf R}}
\def\bL{{\mathbf L}}

\def\pre-tr{\operatorname{pre-tr}}

\def\Hom{\operatorname{Hom}}
\def\End{\operatorname{End}}

\newcommand{\bbA}{{\mathbb A}}

\newcommand{\bbZ}{{\mathbb Z}}

\newcommand{\bbP}{{\mathbb P}}
\newcommand{\bbN}{{\mathbb N}}

\newcommand{\cY}{{\mathcal Y}}

\newcommand{\cO}{{\mathcal O}}

\newcommand{\cM}{{\mathcal M}}

\newcommand{\cD}{{\mathcal D}}

\newcommand{\cA}{{\mathcal A}}
\newcommand{\cB}{{\mathcal B}}

\newcommand{\cC}{{\mathcal C}}
\newcommand{\cE}{{\mathcal E}}

\newcommand{\cU}{{\mathcal U}}

\newcommand{\cS}{{\mathcal S}}
\newcommand{\cT}{{\mathcal T}}

\newcommand{\cH}{{\mathcal H}}

\newcommand{\cHom}{{\mathcal Hom}}
\newcommand{\cX}{{\mathcal X}}

\newcommand{\Perf}{\operatorname{Perf}}

\newcommand{\supp}{\operatorname{Supp}}
\newcommand{\Ker}{\operatorname{Ker}}

\newcommand{\Ext}{\operatorname{Ext}}

\newcommand{\Id}{\operatorname{Id}}

\newcommand{\Ind}{\operatorname{Ind}}
\newcommand{\Res}{\operatorname{Res}}

\newcommand{\id}{\operatorname{id}}

\newcommand{\dg}{\operatorname{dg}}

\newcommand{\Mod}{\operatorname{Mod}}
\newcommand{\op}{\operatorname{op}}
\newcommand{\red}{\operatorname{red}}
\newcommand{\ns}{\operatorname{ns}}
\newcommand{\sg}{\operatorname{sg}}

\newcommand{\hf}{\operatorname{hf}}
\newcommand{\chf}{\operatorname{chf}}
\newcommand{\Fid}{\operatorname{Fid}}
\newcommand{\pt}{\operatorname{pt}}
\newcommand{\lhf}{\operatorname{lhf}}
\newcommand{\lchf}{\operatorname{lchf}}
\newcommand{\Coker}{\operatorname{Coker}}
\newcommand{\qc}{\operatorname{qc}}
\newcommand{\Inj}{\operatorname{Inj}}
\newcommand{\noeth}{\operatorname{noeth}}

\title[Categorical resolution of singularities]
{Categorical resolution of singularities}

\author{Valery A.~Lunts}
\address{Department of Mathematics, Indiana University,
Bloomington, IN 47405, USA} \email{vlunts@indiana.edu}

\begin{document}

\begin{abstract} Building on the concept of a smooth DG algebra we
define the notion of a smooth derived category. We then propose the
definition of a categorical resolution of singularities. Our main
example is the derived category $D(X)$ of quasi-coherent sheaves on
a scheme $X$. We prove that $D(X)$ has a canonical categorical
resolution if the base field is perfect and $X$ is a separated
scheme of finite type with a dualizing complex.
\end{abstract}

\thanks{This research was supported in part by NSF grant 48-294-16.}

\maketitle

\tableofcontents

\section{Introduction}

There is a good notion of smoothness for DG algebras. Namely, a DG
algebra $A$ is smooth if it is perfect as a DG $A^{\op}\otimes
A$-module. If $A$ is derived equivalent to a DG algebra $B$ then $A$
is smooth if and only if $B$ is such. Therefore it makes sense to
define smoothness of the derived category $D(A)$ of DG $A$-modules.
This also allows one to discuss smoothness of cocomplete
triangulated categories $T$ which have a compact generator (and come
from a DG category). For example $T$ may be the derived category of
quasi-coherent sheaves on a quasi-compact separated scheme. If $k$
is a perfect field and $X$ is a separated $k$-scheme essentially of
finite type, then $X$ is regular if and only if the category
$D(X)=D(QcohX)$ is smooth.

For any DG algebra $B$ one may view the full subcategory $\Perf
(B)\subset D(B)$ as a "dense smooth subcategory" of $D(B).$ So it is
natural to define (Definition 4.1) a {\it categorical resolution} of
$D(B)$ as a pair $(A,X)$, where $A$ is a smooth DG algebra and $X$
is a DG $B^{\op}\otimes A$-module such that the restriction of the
functor
$$(-)\stackrel{\bL}{\otimes }_BX:D(B) \to D(A)$$
to the subcategory $\Perf(B)$ is full and faithful.

In this paper we give examples of categorical resolutions. In
particular we show that the Koszul duality functor is sometimes a
categorical resolution (Proposition 5.6).

Our main example is the derived category $D(X)$ of quasi-coherent
sheaves on a scheme $X.$ If $\tilde{X}\stackrel{\pi}{\to}X$ is the
usual resolution of singularities, then $\bL \pi ^*:D(X)\to
D(\tilde{X})$ is a categorical resolution if and only if $X$ has
rational singularities. This may suggest that our definition of
categorical resolution is not the right one. However we believe that
this definition still makes sense and that a categorical resolution
of $D(X)$ may in a sense be "better" than the usual $D(\tilde{X}).$
(For example a categorical resolution of $D(X)$ exists for many
nonreduced schemes $X$.)

We show that if $k$ is a perfect field, then for any separated
$k$-scheme $X$ of finite type that has a dualizing complex there
exists a categorical resolution (Theorem 6.3). The corresponding
"resolving" smooth DG algebra $A$ is derived equivalent to
$A^{\op},$ but usually has unbounded cohomology. This is a canonical
categorical resolution of $D(X);$ it has the flavor of Koszul
duality. (After this paper was written we learned that the
smoothness of this DG algebra $A$ was conjectured by Kontsevich.) It
was pointed to us by Van den Bergh that our result implies the
smoothness of the unbounded homotopy category of injectives $K(\Inj
X)$ which was studied by Krause in [Kr]. We discuss this in the last
section.

In a forthcoming paper [Lu2] we propose categorical resolutions of
$D(X)$ of a different kind. Namely we construct new smooth
categories by "glueing" smooth schemes. This is an extension of the
work [Lu1].

It is our pleasure to thank Michel Van den Bergh, Mike Mandell,
Bernhard Keller and Michael Artin for answering many question. We
are also grateful to participants of the seminar on Algebraic
Varieties at the Steklov Institute, where these ideas were
presented. Dmitri Orlov pointed out to me the results in [Rou] and
Dmitri Kaledin informed me of the paper [Ku] in which a similar
notion appears but the approach is different. Alexander Kuznetsov
drew my attention to the recent preprint [BuDr], where a categorical
resolution is constructed for projective curves with only nodes and
cusps as singularities. (As is pointed out in [BuDr], in some cases
this resolution coincides with the one constructed in [Lu1].)

After our talk in Banff in October 2008 Osamu Iyama suggested a
connection with Auslander algebras, but we did not work it out in
this paper.

\section{Triangulated categories, DG categories, compact object}

This section contains some preliminaries.

Fix a field $k.$ All categories are assumed to be $k$-linear and
$\otimes $ means $\otimes _k$ unless mentioned otherwise.

\subsection{Generation of triangulated categories}
Fix a triangulated category $T.$

Let $I$ be a full subcategory of $T.$ We denote by $\langle
I\rangle$ the smallest strictly full subcategory of $T$ containing
$I$ and closed under finite direct sums, direct summands and shifts.
We denote by $\overline{I}$ the smallest strictly full subcategory
of $T$ containing $I$ and closed under direct sums (existing in $T$)
and shifts.

Let $I_1,I_2$ be two full subcategories of $T.$ We denote by
$I_1*I_2$ the strictly full subcategory of objects $M$ such that
there exists an exact triangle $M_1\to M\to M_2$ with $M_i\in T_i.$
Put $I_1\diamond I_2 =\langle I_1*I_2\rangle.$

Define $\langle I\rangle _0=0$ and then define by induction $\langle
I\rangle _i=\langle I\rangle _{i-1}\diamond \langle I\rangle $ for
$i\geq 1.$ Put $\langle I\rangle _\infty=\bigcup _{i\geq 0}\langle
I\rangle _i.$

The objects of $\langle I\rangle _i$ are the direct summands of the
objects obtains by taking an $i$-fold extension of finite direct
sums of objects of $I$ ([BoVdB],2.2).

\begin{defi} We say that

\begin{itemize}
\item $I$ {\rm generates} $T$ if given $C\in T$ with $\Hom
(D[i],C)=0$ for all $D\in I$ and all $i\in \bbZ,$ then $C=0.$
\item $I$ {\rm classically generates} $T$ if $T=\langle I\rangle
_\infty.$
\item An object $D\in T$ is a {\rm strong classical generator} for
$T$ if $\langle I\rangle _d=T$ for some $d \in \bbN.$
\end{itemize}
\end{defi}

\subsection{Cocomplete triangulated categories and compact objects}
A triangulated category $T$ is called {\it cocomplete} if it has
arbitrary direct sums. An object $C\in T$ is called {\it compact} if
$\Hom (C,-)$ commutes with direct sums. Denote by $T^c\subset T$ the
full triangulated subcategory of compact objects. $T$ is called {\it
compactly generated} if $T$ is generated by a set of compact
objects. We say that $T$ is {\it Karoubian} if every projector in
$T$ splits. The following theorem summarizes some known facts
([BoNe],[Ne],[Rou]).

\begin{theo} Let $T$ be a cocomplete triangulated category.

a) Then $T$ and $T^c$ are Karoubian.

Assume in addition that $T$ is compactly generated.

b) Then a set of objects $\cE \subset T^c$ classically generates
$T^c$ if and only if it generates $T.$

c) If a set of objects $\cE \subset T^c$ generates $T$ then $T$
coincides with the smallest strictly full triangulated subcategory
of $T$ which contains $\cE$ and is closed under direct sums.
\end{theo}

\subsection{DG algebras and their derived categories}
A {\it DG algebra} is a graded unital associative ($k$-) algebra
with a differential $d$ of degree $+1$ satisfying the Leibnitz rule
and such that $d(1)=0.$ A homomorphism of DG algebras is a degree
zero $k$-linear homomorphism (not necessarily unital) of graded
associative rings which commutes with the differential. DG algebras
$A$ and $B$ are quasi-isomorphic if there exist a diagram of DG
algebras and homomorphisms
$$A\leftarrow A_1\rightarrow ...\leftarrow A_n \rightarrow B,$$
where all arrows are quasi-isomorphisms.

Let $A$ be a DG algebra. Denote by $A\text{-mod}$ the DG category
([Ke1]) of unital {\it right} DG $A$-modules. For $M,N\in
A\text{-mod}$ we have the complex $\Hom (M,N)=\oplus _{n\in
\bbZ}\Hom^n(M,N),$ where $\Hom ^n(M,N)$ consists of degree $n$
homogeneous homomorphisms of graded modules over the {\it graded}
algebra $A.$ Let $Ho(A)=Ho(A\text{-mod})$ be the {\it homotopy}
category of $A\text{-mod},$ in which we replace the $\Hom$-complexes
by the cohomology in degree zero. This is a triangulated category
and we denote by $D(A)$ the derived category of $A,$ which is the
Verdier localization of $Ho(A)$ with respect to quasi-isomorphisms.
The categories $Ho(A)$ and $D(A)$ are cocomplete and the
localization functor $Ho(A)\to D(A)$  preserves direct sums.

A DG $A$-module $S$ is called h-injective (resp. h-projective) if
for every acyclic DG $A$-module $M$ the complex $\Hom (M,S)$ is
acyclic (resp. $\Hom (S,M)$ is acyclic). There are enough
h-injectives and h-projectives in $A\text{-mod:}$ for every $M\in
A\text{-mod}$ there exist quasi-isomorphisms $M\to I,$ $P\to M,$
where $I$ is h-injective and $P$ is h-projective. Denote by
$I(A),P(A)\subset A\text{-mod}$ the full DG subcategories consisting
of h-injectives and h-projectives respectively. The induced
triangulated functors $Ho(I(A))\to D(A),$ $Ho(P(A))\to D(A)$ are
equivalences. One uses h-injectives and h-projectives to define
right and left derived functors in the usual way.

Let $\phi :A\to B$ be a homomorphism (not necessarily unital) of DG
algebras. Denote $\phi (1_A)=e$. We have the adjoint DG functors of
extension and restriction of scalars
$$\phi
^*(-)=(-)\otimes_AB=(-)\otimes_AeB :A\text{-mod}\to B\text{-mod}$$
$$ \phi _*(-)=\Hom (eB,-):B\text{-mod}\to A\text{-mod}$$
and the induced triangulated functors $\phi ^* :Ho(A)\to Ho(B),$
$\phi _*:Ho (B)\to Ho(A).$ Define the derived functor $\bL \phi
^*:D(A)\to D(B)$ using h-projectives. So $(\bL \phi ^*,\phi _*)$ is
an adjoint pair of functors between $D(A)$ and $D(B).$ If $\phi$ is
a quasi-isomorphism, then $(\bL \phi ^*,\phi _*)$ is a pair of
mutually inverse equivalences. Sometimes the functors $\phi ^*$ and
$\phi_*$ are denoted by $\Ind $ and $\Res $ respectively.

Denote by $\Perf (A)\subset D(A)$ the full triangulated subcategory
which is classically generated by the DG $A$-module $A.$ We call
objects of $\Perf (A)$ the {\it perfect} DG $A$-modules. Note that a
the functor $\bL \phi ^*$ as above preserves perfect modules (even
though $\bL \phi ^*(A)\neq B$ when $\phi $ is not unital).

For any $M\in D(A)$ we have $\Hom _{Ho(A)}(A,M)=\Hom
_{D(A)}(A,M)=H^0(M).$ Thus $A$ is a generator for $D(A).$ Since
$H^0(-)$ commutes with direct sums, the object $A\in D(A)$ is
compact. Hence $\Perf (A)\subset D(A)^c.$

\begin{prop} [Ke1] $\Perf (A)=D(A)^c.$
\end{prop}

The following definition extends the notion of Morita equivalence to
DG algebras.

\begin{defi} DG algebras $A$ and $B$ are called {\rm derived equivalent}
 if there exists a DG $A^{\op}\otimes B$-module $K$ such that the
 functor $-\stackrel{\bL}{\otimes }_AK:D(A)\to D(B)$ is an
 equivalence of categories.
 \end{defi}

For example, if $\phi :A\to B$ is a quasi-isomorphism of DG algebras
 then $A$ and $B$ are derived equivalent ($K=B$).

\subsection{Derived categories of abelian Grothendieck categories}
Let $\cA$ be an abelian category, $C(\cA)$ the abelian category of
complexes over $\cA,$ $Ho(\cA),$ $D(\cA)$ - the corresponding
homotopy and derived categories. One can make $C(\cA)$ into a DG
category $C^{\dg}(\cA)$ in the usual way: given $M,N\in C(\cA)$ we
get the complex $\Hom (M,N)=\oplus _{n\in \bbZ}\Hom ^n(M,N),$ where
$\Hom ^n(M,N)=\prod _{i\in \bbZ}\Hom (M^i,N^{i+n}).$ Then
$Ho(C^{\dg}(\cA))=Ho(\cA).$

An object $I\in C(\cA)$ is called h-injective if for every acyclic
$M\in C(\cA)$ the complex $\Hom (M,I)$ is acyclic. Denote by
$I(\cA)\subset C^{\dg}(\cA)$ the full DG category of h-injectives.

Recall that an object $G\in \cA$ is called a g-object if the functor
$X\mapsto \Hom _{\cA}(G,X)$ is conservative, i.e. $X\to Y$ is an
isomorphism as soon as $\Hom (G,X)\to \Hom (G,Y)$ is an isomorphism.
Such an object $G$ is usually called a generator, but we already
used this term in Definition 2.1 in a different context.

Recall that an abelian category $\cA$ is called a {\it Grothendieck
category} if it has a g-object, small inductive limits and the
filtered inductive limits are exact. In particular $\cA$ has
arbitrary direct sums.

If $\cA$ is a Grothendieck category, then so is $C(\cA).$ Then the
categories $Ho(\cA),$ $D(\cA)$ are cocomplete and the natural
functors $C(\cA)\to Ho(\cA)\to D(\cA)$ preserve direct sums. The
following proposition is proved for example in [Ka-Sch], Thm.
14.1.7.

\begin{prop} Let $\cA$ be a Grothendieck category. Then for every
$M\in C(\cA)$ there exists a quasi-isomorphism $M\to I,$ where $I\in
C(\cA)$ is h-injective. Thus the triangulated category $Ho(I(\cA))$
is equivalent to $D(\cA).$ (Hence in particular the bi-functor $\bR
\Hom (-,-):D(\cA)^{\op}\times D(\cA)\to D(k)$ is defined.)
\end{prop}

Derived categories (admitting a compact generator) of Grothendieck
categories can be described using DG algebras. The proof of the
following proposition is the same argument as in [Ke1],Lemma 4.2. We
present in here because it will be used again later.

\begin{prop} Let $\cA$ be a Grothendieck category such that the
triangulated category $D(\cA)$ has a compact generator $E.$ Denote
by $A$ the DG algebra $\bR \Hom (E,E).$ Then the functor $\bR \Hom
(E,-):D(\cA) \to D(A)$ is an equivalence of categories.
\end{prop}

\begin{proof} Since $Ho(I(\cA))\simeq D(\cA)$ we may assume that
$E$ is h-injective and hence $A=\Hom (E,E).$ Define the DG functor
$$I(\cA)\to A\text{-mod},\quad M\mapsto \Hom (E,M).$$
Let $\Psi _E:Ho(I(\cA))\to D(A)$ be the composition of the induced
functor $Ho(I(\cA))\to Ho(A)$ with the localization $Ho(A)\to D(A).$

 Let us prove that
$\Psi _E$ is full and faithful.

Let $T\subset Ho(I(\cA))$ be the full triangulated subcategory of
objects $M$ such that the map
$$\Hom (E,M[n])\to \Hom (\Psi _E(E),\Psi _E(M[n]))$$
is an isomorphism for all $n\in \bbZ.$ Then $T$ contains $E$ and is
closed under direct sums. Hence $T=Ho(I(\cA))$ by Theorem 2.2c).
Similarly let $S\subset Ho(I(\cA))$ be the full triangulated
category consisting of objects $N$ such that for each $M\in
Ho(I(\cA))$ the map
$$\Hom (N,M)\to \Hom (\Psi _E(N),\Psi _E(M))$$
is an isomorphism. Then $S$ contains $E$ and is closed under direct
sums. So $S=Ho(I(\cA)).$

The fully faithful triangulated functor $\Psi _E$ preserves direct
sums and takes the compact generator $E$ to the compact generator
$A.$ Since categories $Ho(I(\cA))$ and $D(A)$ are cocomplete it
follows from Theorem 2.2c) that $\Psi _E$ is essentially surjective.
\end{proof}

\begin{remark} In the context of Proposition 2.6 let $E^\prime$
be another compact generator of $D(\cA)$ with $A^\prime =\bR \Hom
(E^\prime,E^\prime).$ Then the DG algebras $A$ and $A^\prime$ are
derived equivalent. Indeed assume that $E$ and $E^\prime$ are
h-injective and consider the DG $A^{op}\otimes A^{\prime}$-module
$\Hom (E^\prime,E).$ Then using the notation in the proof of
Proposition 2.6 we have the obvious morphism of functors
$$\mu :\Psi _E(-)\stackrel{\bL}{\otimes}_A \Hom
(E^\prime,E)\to \Psi_{E^\prime}(-).$$ Both functors preserve direct
sums and $\mu (E)$ is an isomorphism. Hence $\mu$ is an isomorphism
(Theorem 2.2c). But $\Psi _E$ and $\Psi _{E^\prime}$ are
equivalences. Hence
$$(-)\stackrel{\bL}{\otimes }_A\Hom  (E^\prime,E):D(A)\to
D(A^\prime)$$ is also an equivalence. In fact it is easy to see
(using Lemma 2.14) that the DG algebras $A$ and $A^\prime$ are
quasi-isomorphic.
\end{remark}

Actually, Proposition 2.6 is a special case of the following general
theorem of Keller ([Ke1],Thm.4.3).

\begin{theo} Let $\cE$ be a Frobenius exact category. Assume that
the corresponding triangulated stable category $\underline{\cE}$ is
cocomplete and has a compact generator. Then $\underline{\cE}\simeq
D(A)$ for a DG algebra $A.$
\end{theo}

\begin{remark} As in Remark 2.7 one can show that the DG algebra $A$
in  Theorem 2.8 is well defined up to a derived equivalence.
\end{remark}

Triangulated categories which are equivalent to the stable category
$\underline{\cE}$ of a Frobenius exact category are called {\it
algebraic} in [Ke2]. For example derived categories of abelian
categories are algebraic.

\subsection{Schemes} Let $X$ be a $k$-scheme. We denote by $QcohX$ the abelian
category of quasi-coherent sheaves on $X.$ Put $D(X)=D(QcohX)$ and
denote by $\Perf (X)\subset D(X)$ the full subcategory of perfect
complexes (i.e. complexes which are locally quasi-isomorphic to a
finite complex of free $\cO_X$-modules of finite rank).

If $X$ is quasi-compact and quasi-separated, then $QcohX$ is a
Grothendieck category [ThTr], Appendix B.

The first assertion in the next theorem is due to Neeman  and the
second is in [BoVdB]

\begin{theo} Let $X$ be a quasi-compact and separated
scheme. Then

a) $D(X)^c=\Perf (X).$

b) The category $D(X)$ has a compact generator.
\end{theo}

\begin{cor} Let $X$ be a quasi-compact separated scheme. Then there
exists a DG algebra $A,$ such that $D(X)\simeq D(A).$
\end{cor}

\begin{proof} Indeed, since $QcohX$ is a Grothendieck category  the
corollary follows from Proposition 2.6 and Theorem 2.10b).
\end{proof}

Thus many triangulated categories "in nature" look like $D(A)$ or
$\Perf (A)$ for a DG algebra $A.$

\subsection{A few lemmas}

\begin{lemma} Let $A$ and $B$ be DG algebras,
$M\in A^{\op}\otimes B\text{-mod}$ such that the functor
$$\Phi _M(-):=(-)\stackrel{\bL}{\otimes}_AM:D(A)\to D(B)$$
induces an equivalence of full subcategories $\Perf
(A)\stackrel{\sim}{\to} \Perf (B).$ Then $\Phi _M$ is an
equivalence. In particular $A$ and $B$ are derived equivalent.
\end{lemma}

\begin{proof} The DG $A$-module  is a classical generator of
$\Perf (A).$ Hence the object $\Phi _M(A)$ is a classical generator
for $\Perf (B),$ and therefore by Proposition 2.3 and Theorem 2.2b)
it is a compact generator for $D(B).$  Thus the functor $\Phi _M$
has the following three properties:

a) it preserves direct sums;

b) it maps a compact generator $A$ to a compact generator $\Phi
_M(A);$

c) it induces an isomorphism $\Ext ^\bullet (A,A)\stackrel{\sim}{\to
}\Ext ^\bullet (\Phi_M(A),\Phi _M(A)).$

Using the same argument as in the proof of Proposition 2.6 it
follows easily from a),b),c) that $\Phi _M$ is an equivalence.
\end{proof}

\begin{lemma} Let $A$ and $B$ be DG algebras and $F:D(A)\to D(B)$ be
a triangulated functor with the following properties

a) $F(\Perf (A))\subset \Perf (B).$

b) The restriction of $F$ to $\Perf (A)$ is full and faithful.

c) $F$ preserves direct sums.

Then $F$ is full and faithful.
\end{lemma}

\begin{proof} Same argument as in the proof of Proposition 2.6 and
Lemma 2.12.
\end{proof}

Let $\cA$ be an abelian category, $X,Y\in C(\cA)$ and $f:X\to Y$ a
morphism of complexes. Consider the cone $C_f\in C(\cA)$ of the
morphism $f$ and the DG algebra $\End(C_f).$ Let $\cC\subset
\End(C_f)$ be the DG subalgebra which preserves the complex $Y,$
$$\cC =\left( \begin{array}{cc}
              \End (Y) & \Hom (X[1],Y)\\
              0 & \End (X[1])
              \end{array}\right)$$
with the projections $p_X: \cC \to \End (X[1]),$ $p_Y:\cC \to \End
(Y).$ More generally, let $A\to \End (X)=\End (X[1])$ be a
homomorphism of DG algebras. Then we can consider the corresponding
DG algebra
$$\cC _A = \left( \begin{array}{cc}
              \End (Y) & \Hom (X[1],Y)\\
              0 & A
              \end{array}\right)$$
with the projections $p_A:\cC _A\to A$ and $p_Y:\cC _A \to \End
(Y).$

\begin{lemma} Assume that the induced map $f^*:\End(Y)\to \Hom
(X,Y)$ and the composition $A\to \End (X)\stackrel{f_*}{\to}\Hom
(X,Y)$ are quasi-isomorphisms. Then $p_A$ and $p_Y$ are
quasi-isomorphisms. In particular the DG algebras $A$ and $\End (Y)$
are quasi-isomorphic.
\end{lemma}

\begin{proof} Indeed, our assumptions imply that the kernels $\Ker
p_A= \End (Y)\oplus \Hom (X[1],Y)$ and $\Ker p_Y =A\oplus \Hom
(X[1],Y)$ are acyclic.
\end{proof}

\section{Smooth DG algebras and smooth derived categories}

\begin{defi} (Kontsevich). A DG algebra $A$ is {\rm smooth} if $A\in
\Perf (A^{\op}\otimes A).$
\end{defi}

We thank Bernhard Keller for the following remark.

\begin{remark} If $A$ is smooth, then so is $A^{\op}.$ Indeed, the isomorphism of DG algebras
$$A^{\op}\otimes A\to A\otimes A^{\op},\quad a\otimes b\mapsto
b\otimes a$$ induces an equivalence $D(A^{\op}\otimes A)\simeq
D(A\otimes A^{\op})$ which preserves perfect DG modules and sends
$A$ to $A^{\op}.$
\end{remark}

\begin{lemma} Let $A$ and $B$ be smooth DG algebras. Then so is
$A\otimes B.$
\end{lemma}

\begin{proof} The bifunctor $\otimes :D(A^{\op}\otimes A)\times
D(B^{\op}\otimes B)\to D((A\otimes B)^{\op}\otimes A\otimes B)$ maps
$\Perf (A^{\op}\otimes A)\times \Perf (B^{\op}\otimes B)\to \Perf
((A\otimes B)^{\op}\otimes A\otimes B)$ and sends $(A,B)$ to
$A\otimes B.$
\end{proof}

The next definition is the analogue for DG algebras of the notion of
finite global dimension for associative algebras.

\begin{defi} We say that a DG algebra $A$ is {\rm weakly smooth} if
$D(A)=\langle \overline{A}\rangle _d$ for some $d\in \bbN$
(Definition 2.1). That is every DG $A$-module is quasi-isomorphic to
a direct summand of a $d$-fold extension of direct sums of shifts of
$A.$
\end{defi}

\begin{lemma} Assume that the DG algebra $A$ is weakly smooth,
$D(A)= \langle \overline{A}\rangle _d.$ Then $\Perf (A)=\langle
A\rangle _d.$ In particular $A$ is a strong generator for $\Perf
(A).$
\end{lemma}

\begin{proof} Recall that for any DG $A$-module $M$
$$\Hom _{D(A)}(A,M)=H^0(M).$$
Since cohomology commutes with filtered inductive limits of
complexes we have
$$\Hom _{D(A)}(A,\lim_\to M_i)=\lim _\to \Hom _{D(A)}(A,M_i)$$
for any filtered inductive system of DG $A$-modules $\{M_i\}$ (here
the inductive limit is taken in the abelian category of DG
$A$-modules with morphisms being closed morphisms of degree zero).
Hence this holds also for any perfect DG $A$-module instead of $A.$

Fix $P\in \Perf (A).$ By our assumption $P$ (as any DG $A$-module)
is isomorphic to a direct summand of a $d$-fold extension $Q$ of
direct sums of shifts of $A.$ That is we have morphisms
$P\stackrel{i}{\to} Q\stackrel{p}{\to }P,$ such that $p\cdot i=\id.$
Notice that the DG module $Q$ is the union of its DG submodules
$\{Q_j\}$ which are $d$-fold extensions of {\it finite} direct sums
of shifts of $A.$ Hence the morphism $i:P\to Q$ factors through some
$Q_j\subset Q,$ so that the composition $P\stackrel{i}{\to
}Q_j\stackrel{p}{\to }P$ is the identity. Hence $P$ is isomorphic to
a direct summand of $Q_j$, i.e. $P\in \langle A\rangle _d.$
\end{proof}

\begin{lemma} a) Suppose $A$ is smooth. Then it is weakly smooth.

b) Assume that $A$ is smooth and is concentrated in degree zero.
Then $A$ has finite global dimension.
\end{lemma}

\begin{proof} a) Any DG $A^{\op}\otimes A$-module $M$ defines a functor
$F_M:D(A)\to D(A),$ $F_M(-)=(-)\stackrel{\bL}{\otimes}_AM.$ We have
$F_A\simeq \Id _{D(A)}.$ Thus if $A\in \langle A^{\op}\otimes
A\rangle _d,$ then for any $N\in D(A),$ we have $N\simeq F_A(N)\in
\langle \overline{A}\rangle _d.$

b) A perfect DG $A^{\op}\otimes A$-module is a homotopy direct
summand if a bounded complex of free $A ^{\op}\otimes A$-modules (of
finite rank). Thus as in the proof of a) for any $A$-module $M$ the
complex $F_A(M)$ (which is quasi-isomorphic to $M$) is a homotopy
direct summand of a complex of free $A$-modules which is bounded
independently of $M.$ Hence $A$ has finite global dimension.
\end{proof}

\begin{example} Let $A$ be a finite inseparable field extension of
$k.$ Then $A$ is weakly smooth (with $d=1$), but not smooth.
\end{example}

Nevertheless one has the following result.

\begin{prop} Assume that the field $k$ is perfect. Let $A$ and $C$ be
localizations of finitely generated commutative $k$-algebras.

a) Assume that the algebras $A,C$ have finite global dimension. Then
the algebra $A\otimes C$ is also regular (hence so is $A\otimes A$)
and $A$ is a perfect DG $A\otimes A$-module (i.e. the DG algebra $A$
is smooth).

b) Vice versa if $A$ has infinite global dimension, then $A$ is not
a perfect DG $A\otimes A$-module (i.e. the DG algebra $A$ is not
smooth).
\end{prop}

\begin{proof} a). Denote $B:= A\otimes C.$ Since $B$ is noetherian it
suffices to prove that it is regular.

 We need to prove that the localization $B_{\frak{m}}$ of $B$ at every
maximal ideal is a regular local ring. For this we may assume that
$A$ and $C$ are  finitely generated $k$-algebras. Put
$K=B/\frak{m}.$ Then by Nullstellensatz $\dim _kK <\infty.$ It
follows that the ideal $\frak{n}:=\frak{m} \cap (A\otimes 1)\subset
A$ is also maximal. Put $L=A/\frak{n}A;$ this is a finite separable
extension of  $k.$ Consider the obvious (flat) embedding of local
rings $A_{\frak{n}}\to B_{\frak{m}}.$ By Theorem 23.7 in [Ma] it
suffices to prove that the ring
$F:=B_{\frak{m}}/\frak{n}B_{\frak{m}}$ is regular.

Consider the embedding $A=A\otimes 1\hookrightarrow B$ and the
induced quotient $B/\frak{n}B\simeq L\otimes C,$ which is an etale
extension of $C$ (since the field $k$ is perfect). Thus
$B/\frak{n}B$ is a regular ring. But $F$ is a localization of
$B/\frak{n}B$ at (the image of) the ideal $\frak{m}.$ So $F$ is also
regular.

b). Follows from Lemma 3.6b).
\end{proof}

\subsection{Derived invariance of smoothness}

Let us show that smoothness is an invariant of the derived
equivalence class of DG algebras.

\begin{lemma} Assume that $A$ and $B$ are derived equivalent. Then
$A$ is smooth if and only if $B$ is smooth.
\end{lemma}

\begin{proof} For $M\in D(A^{op}\otimes B)$ denote by
$\Phi _M (-):D(A)\to D(B)$ the functor
$(-)\stackrel{\bL}{\otimes}_AM.$  It has the right adjoint functor
$\Psi _M(-):=\bR \Hom _B(M,-)$. Assume that $\Phi _M$ is an
equivalence. Then so is $\Psi _M$, and hence in particular $\Psi _M$
preserves direct sums, i.e. $M$ is compact as a DG $B$-module. But
then we claim that for any $T\in D(B)$ the canonical morphism of DG
$A$-modules
$$T\stackrel{\bL}{\otimes }_B\bR \Hom _B(M,B)\to \bR \Hom _B(M,T)$$
is a quasi-isomorphism. Indeed, since $M$ is compact it suffices to
check the claim for $T=B$ (Theorem 2.2c), where it is obvious. It
follows that the functor $\Psi _M$ is isomorphic to the functor
$$\Phi _N(-)=(-)\stackrel{\bL}{\otimes }_BN, \quad \text{where} \ \
N=\bR \Hom _B (M,B).$$ The isomorphisms of functors
$$\Phi _N \cdot \Phi _M\simeq \Id, \quad \Phi _M \cdot \Phi _N\simeq
\Id$$ induce in particular the quasi-isomorphisms of DG
$A^{\op}\otimes A$- and $B^{\op}\otimes B$-modules respectively
$$M\stackrel{\bL}{\otimes }_BN\simeq A,\quad N\stackrel{\bL}{\otimes
}_AM\simeq B.$$

Now consider the functors
$${}_N\Delta _M(-):=N\stackrel{\bL}{\otimes
}_A(-)\stackrel{\bL}{\otimes}_AM:D(A^{op}\otimes A)\to
D(B^{op}\otimes B),$$
$${}_M\Delta _N(-):=M\stackrel{\bL}{\otimes
}_B(-)\stackrel{\bL}{\otimes}_BN:D(B^{op}\otimes B)\to
D(A^{op}\otimes A).$$ The quasi-isomorphisms above imply the
isomorphisms of functors
$${}_M\Delta _N\cdot {}_N\Delta _M\simeq \Id, \quad
{}_N\Delta _M\cdot {}_M\Delta _N\simeq \Id.$$ Hence ${}_M\Delta _N$
and ${}_N\Delta _M$ are mutually inverse equivalences. In particular
they preserve compact objects, i.e. perfect complexes. But notice
that ${}_N\Delta _M(A)\simeq B$. This proves the lemma.
\end{proof}

\begin{cor} Assume that the DG algebras $A$ and $B$ are
quasi-isomorphic. Then $A$ is smooth if and only if $B$ is smooth.
\end{cor}

\begin{proof} We may assume that there exists a
quasi-isomorphism $\phi :A\to B$ of DG algebras. Then the functor
$$(-)\stackrel{\bL}{\otimes }_AB:D(A)\to D(B)$$
is an equivalence of categories. So we are done by Lemma 3.9.
\end{proof}

\subsection{Gluing smooth DG algebras}

Let $A$ and $B$ be DG algebras and $N\in A^{\op}\otimes
B\text{-mod}.$ Then we obtain a new DG algebra
$$C=\left( \begin{array}{cc}
B & 0\\
N & A\end{array} \right).$$

\begin{prop} Assume that the DG algebras $A$ and $B$ are smooth.
 Also assume that $N\in \Perf (A^{\op}\otimes
B).$ Then $C$ is smooth.
\end{prop}

\begin{proof}
Since quasi-isomorphic DG algebras are derived equivalent we may
assume that the DG $A^{\op}\otimes B$-module $N$ is h-projective
(hence it is also h-projective as DG $A^{\op}$- or $B$-module).

If $D$ and $E$ are DG algebras we will denote by $M_E,$ ${}_DM,$
${}_DM_E$ respectively a DG $E$-, $D^{\op}$- $D^{\op}\otimes
E$-module.

 It is easy to see that a DG $C$-module is the same as
a triple $S=(S_A,S_B,\phi _S:S_A\otimes _AN\to S_B),$ where
$S_A,S_B$ are DG $A$- and $B$-modules respectively and $\phi _S$ is
a closed degree zero morphism of DG $B$-modules.

Similarly, a DG $C^{\op}\otimes C$-module is given by the following
data
$$\begin{array}{rcl}
M & = & \{ {}_BM_A,{}_AM_A,{}_BM_B,{}_AM_B;\\
  &   & {}_B\Theta _{AB}:({}_BM_A)\otimes _AN\to {}_BM_B,\\
 &   & {}_A\Theta _{AB}:({}_AM_A)\otimes _AN\to {}_AM_B,\\
 &   &  {}_{BA}\Theta _{A}:N\otimes_B({}_BM_A)\to {}_AM_A,\\
 &   & {}_{BA}\Theta _{B}:N\otimes _B({}_BM_B)\to {}_AM_B \}\\
\end{array}
$$
where all the $\Theta $'s are closed degree zero morphisms of the
corresponding DG modules, such that the diagram
$$\begin{array}{rcl}
N\otimes _B({}_BM_A)\otimes _AN & \stackrel{\id\otimes ({}_B\Theta
_{AB})}{\longrightarrow} & N\otimes _B({}_BM_B)\\
{}_{BA}\Theta _A\otimes \id \downarrow & & \downarrow {}_{BA}\Theta _B\\
{}_AM_A\otimes _AN & \stackrel{{}_A\Theta _{AB}}{\longrightarrow} &
{}_AM_B
\end{array}$$
commutes. It is convenient to describe such DG $C^{\op}\otimes
C$-module $M$ symbolically by a diagram
$$\begin{array}{rcl}
{}_BM_A & \stackrel{{}_B\Theta _{AB}}{\longrightarrow} & {}_BM_B\\
{}_{BA}\Theta _A\downarrow & & \downarrow {}_{BA}\Theta _B\\
{}_AM_A & \stackrel{{}_A\Theta _{AB}}{\longrightarrow} & {}_AM_B
\end{array}$$

Then the diagram corresponding to the diagonal DG module $C$ is
$$\begin{array}{ccl}
0 & \to & B \\
\downarrow & & \downarrow \id \\
A & \stackrel{\id}{\to} & N
\end{array}
$$

We have the obvious (non-unital) inclusions of DG algebras
$A^{\op}\otimes A\to C^{\op}\otimes C,$ $A^{\op}\otimes B\to
C^{\op}\otimes C,$ etc. Hence the corresponding DG functors of
 extension of scalars
$$\Ind _{A^{\op}\otimes A}:A^{\op}\otimes A\text{-mod}\to
C^{\op}\otimes C\text{-mod}, ... $$ Consider the corresponding
derived functors $\bL \Ind _{A^{\op}\otimes A}:D(A^{\op}\otimes
A)\to D(C^{\op}\otimes C),...$ They preserve perfect DG modules.

Consider the diagonal DG $A^{\op}\otimes A$-module $A.$ Then
$$\begin{array}{rcl}
\bL \Ind_{A^{\op}\otimes A}(A) & =
 & A\stackrel{\bL}{\otimes}_{A^{\op}\otimes A}(C^{\op}\otimes C)\\
  & = & A\stackrel{\bL}{\otimes}_{A^{\op}\otimes A}[(A^{\op}\otimes A)\oplus
  (A^{\op}\otimes N)]\\
     & = & A\oplus N.
\end{array}$$
 Thus $\bL \Ind_{A^{\op}\otimes A}(A)$ is quasi-isomorphic to the
DG $C^{\op}\otimes C$-module
$$\begin{array}{ccc}
0 & \to & 0 \\
\downarrow & & \downarrow  \\
A & \stackrel{\id}{\to} & N
\end{array}
$$
Similarly, $\bL \Ind_{B^{\op}\otimes B}(B)$ is quasi-isomorphic to
$$\begin{array}{ccl}
0 & \to & B \\
\downarrow & & \downarrow \id \\
0 & {\to} & N.
\end{array}
$$
Also $\bL \Ind_{A^{\op}\otimes B}(N)$ is equal to
$$\begin{array}{ccc}
0 & \to & 0 \\
\downarrow & & \downarrow \\
0 & {\to} & N,
\end{array}
$$
We conclude that the diagonal DG $C^{\op}\otimes C$-module $C$ is
quasi-isomorphic to the cone of the obvious morphism
$$\bL \Ind_{A^{\op}\otimes B}(N)\to \bL
\Ind_{A^{\op}\otimes A}(A)\oplus \bL \Ind_{B^{\op}\otimes B}(B).$$
Thus our assumptions on $A,B,$ and $N$ imply that $C$ is perfect.
\end{proof}

\subsection{Smoothness for schemes}

Next we show that for nice schemes the two notions of smoothness
coincide.

\begin{defi} A ($k$-) scheme $Y$ is {\rm essentially of finite type}
if $Y$ is a separated scheme which admits a finite open covering by
affine schemes $SpecC$, where $C$ is a localization of a finitely
generated $k$-algebra. In particular it is quasi-compact.
\end{defi}

\begin{prop} Assume that the field $k$ is perfect.
 Let $X$ be a scheme which is essentially of finite type. Let
$E\in \Perf (X)$ be a compact generator of $D(X),$ i.e. the functor
$F:D(X)\to D(A)$, $F(M)=\bR \Hom (E,M)$ is an equivalence, where
$A=\bR \Hom (E,E)$ (Proposition 2.6, Theorem 2.10, Corollary 2.11).
Then $X$ is a regular scheme if and only if the DG algebra $A$ is
smooth.
\end{prop}

\begin{proof} Note that Proposition 3.8 provides a local version
of this proposition. Indeed, if $X=SpecC$ then $\cO _X$ is a compact
generator of $D(X),$ so that $D(X)=D(C)$ (Serre's theorem).

Notice that the contravariant functor $M\mapsto M^*:= \bR \cHom
(M,\cO _X)$ is an auto-equivalence of the category $\Perf (X).$ It
follows that $E^*$ is also a generator of $D(X)$.

Moreover the following result implies that $E^*\boxtimes E\in \Perf
(X\times X)$ is a compact generator for $D(X\times X).$

\begin{lemma} Let $Y$ and $Z$ be quasi-compact
separated schemes. Assume that $S\in \Perf(Y)$, $T\in \Perf(Z)$ are
the compact generators of $D(Y)$ and $D(Z)$ respectively. Then
$S\boxtimes T$ is a compact generator of $D(Y\times Z)$
\end{lemma}

\begin{proof} It is [BoVdB], Lemma 3.4.1.
\end{proof}

\begin{lemma} There exist  canonical quasi-isomorphisms of DG
algebras

a) $\bR \Hom (E^*,E^*)\simeq A^{\op},$

b) $\bR \Hom (E^*\boxtimes E,E^*\boxtimes E)\simeq A^{\op}\otimes
A.$

Let $\Delta :X\to X\times X$ be the diagonal closed embedding.

c) There exists a canonical equivalence of categories $D(X\times
X)\to D(A^{\op}\otimes A)$ which takes the object $\Delta _*\cO _X$
to the diagonal DG $A^{\op}\otimes A$-module $A.$
\end{lemma}

\begin{proof} The proof is essentially the same as that of
Proposition 6.17 below. We omit it.
\end{proof}

It follows from part c) of Lemma 3.15 that $\Delta _*\cO _X\in \Perf
(X\times X)=D(X\times X)^c$ if and only if $A\in \Perf
(A^{\op}\otimes A)=D(A^{\op}\otimes A)^c.$ If $X$ is regular, then
$X\times X$ is also regular by Proposition 3.8a) hence
$D^b(coh(X\times X))=\Perf (X\times X),$ so in this case $A$ is
smooth.

Vice versa, assume that $X$ is not regular. It suffices to prove
that $\Delta _*\cO _X$ is not in  $\Perf (X\times X).$ The question
is local, so we may assume that $X=SpecC,$ where $C$ is a
localization of a finitely generated $k$-algebra. Then $C$ has
infinite global dimension and by Proposition 3.8b) we know that $C$
is not a perfect DG $C\otimes C$-module.
\end{proof}

\subsection{Smooth triangulated categories} Let $T$ be a cocomplete
triangulated category with a compact generator. We would like to say
that $T$ is smooth if there exists an equivalence of triangulated
categories $T\simeq D(A),$ where $A$ is a smooth DG algebra.
However, we don't know if this is well defined, because there exist
DG algebras which are not derived equivalent, but their derived
category are equivalent as triangulated categories. So the
triangulated category $T$ should come with an {\it enhancement},
i.e. some DG category. For example, $T$ maybe the derived category
of an abelian Grothendieck category or the stable category of a
Frobenius exact category. Then using Proposition 2.6, Theorem 2.8
and Remarks 2.7, 2.9 we may define the notion of smoothness for $T.$

\begin{defi}
a) Let $A$ be a DG algebra. We call its derived category $D(A)$ {\rm
smooth} if $A$ is smooth.

b) Let $\cA$ be an abelian Grothendieck category such that the
derived category $D(\cA)$ has a compact generator $K$. Denote $A=\bR
\Hom (K,K),$ so that $D(\cA)\simeq D(A)$ (Proposition 2.6). Then
$D(\cA)$ is called {\rm smooth} if $A$ is smooth.

c) Let $\cE$ be an exact Frobenius category such that the stable
category $\underline{\cE}$ is cocomplete and has a compact
generator. Then $\underline{\cE}\simeq D(A)$ for a DG algebra $A$
(Theorem 2.8). We call $\underline{\cE}$ {\rm smooth} if $A$ is
smooth.

Note that b) and c) are well defined by Remarks 2.7,2.9.
\end{defi}

Note that we have defined smoothness only for "big", i.e. cocomplete
categories.

\section{Definition of a categorical resolution of singularities}

\begin{defi} Let $A$ be a DG algebra. A {\rm categorical resolution}
of $D(A)$ (or of $A$) is a pair $(B,X)$, where $B$ is a smooth DG
algebra and $X\in D(A^{op}\otimes B)$ is such that the restriction
of the functor
$$\theta (-):=(-)\stackrel{\bL}{\otimes}_AX:D(A)\to D(B)$$
 to the subcategory $\Perf (A)$ is full and faithful.
We also call a categorical resolution of $D(A)$ a pair $(B,E),$
where $B$ is a smooth DG algebra and $E\in D(A\otimes B^{\op})$ is
such that the restriction of the functor
$$\theta (-):=\bR \Hom (E,-):D(A)\to D(B)$$
to the subcategory $\Perf (A)$ is full and faithful.

 Sometimes we
 will say that the pair $(D(B), \theta)$, or simply $D(B)$ or $\theta$ is a
 resolution of $D(A)$.
 \end{defi}

Let us try to explain this definition. For {\it any} DG algebra $A$
the perfect DG $A$-modules form (in our opinion) a "smooth dense
subcategory" of $D(A).$ Hence a categorical resolution of $D(A)$
should not change the subcategory $\Perf (A).$

\begin{remark} Let $A$ be a DG algebra and $B$ be a smooth DG
algebra. Let $E$ be a DG $A\otimes B^{\op}$-module such that the
functor $\bR \Hom (E,-):D(A)\to D(B)$ is full and faithful on the
subcategory $\Perf (A).$ Then the functor
$(-)\stackrel{\bL}{\otimes}_A\bR \Hom (E,A):D(A)\to D(B)$ is also a
categorical resolution of singularities. Indeed, there is a natural
isomorphism of functors from $\Perf (A)$ to $D(B)$
$$(-)\stackrel{\bL}{\otimes}_A\bR \Hom (E,A)\to \bR \Hom (E,-).$$
So the existence of two possibilities in Definition 4.1 is only for
convenience.
\end{remark}

\begin{defi} Let $A$ be a DG algebra and $(B,\theta),$ $(B^\prime
,\theta ^\prime)$ two categorical resolutions of $D(A).$ We say that
these resolutions are equivalent if there exists a DG $B^{\op}
\otimes B^\prime$-module $S$ such that the functor $\Phi
_Y(-):=(-)\stackrel{\bL}{\otimes}_BS:D(B)\to D(B^\prime )$ is an
equivalence and the functors $\Phi _Y\cdot \theta $ and $\theta
^\prime$ are isomorphic.
\end{defi}

In the rest of the paper we will discuss some examples of
categorical resolutions.

\section{Miscellaneous examples of categorical resolutions}

\begin{example} Assume that $k$ is a perfect field. Let $X$ be an algebraic variety
over $k$ and $\pi :\tilde{X}\to X$ its resolution of singularities.
Then by Proposition 3.13 the category $D(\tilde{X})$ is smooth. The
pair $(D(\tilde{X}), \bL \pi ^*)$ is a categorical resolution of
$D(X)$ if and only if the adjunction morphism
$$\phi (M): M\to \bR \pi _*\bL \pi ^*(M)$$
is a quasi-isomorphism for every $M\in \Perf (X)$. This question is
local on $X$, so it suffices to check if the morphism $\phi (\cO
_X)$ is a quasi-isomorphism. We conclude that $(D(\tilde{X}), \bL
\pi ^*)$ is a categorical resolution of $D(X)$ if and only if $X$
has rational singularities.
\end{example}

The above example may suggest that our definition of categorical
resolution of singularities is not the right one because it is
consistent with the usual geometric resolution only in the case of
rational singularities. To make things even worse let us note that
if a morphism of varieties $Y\to X$ defines a categorical resolution
of $D(X)$, then so does the morphism $\bbP ^n\times Y\to X$.
Nevertheless, in this paper we want to argue that our definition
makes sense. In particular, we will show that even if $X$ has
nonrational singularities (and the field $k$ has positive
characteristic!) there exists a categorical resolution of $D(X).$

\begin{example} Assume that  ${\rm char}(k)=0.$
Let $R$ be a commutative finitely generated $k$-algebra, such that
$Y=SpecR$ is smooth. Let $G$ be a finite group acting on $Y$ and
denote by $R *G$ the corresponding crossed product algebra. It is
smooth. Consider the possibly singular scheme $Y//G:=Spec R^G.$ Then
the functor
$$R\stackrel{\bL}{\otimes }_{R^G}(-):D(R^G)\to D(R * G)$$
is a categorical resolution of singularities. Note that $D(R^G)=
D(Y//G)$ and $D(R * G)$ is equivalent to the derived category of
$G$-equivariant quasi-coherent sheaves on $Y.$
\end{example}

\begin{example} [VdB] Let $k$ be algebraically closed and $R$ be a an
integral commutative Gorenstein $k$-algebra. Let $M$ be a reflexive
$R$-module such that the algebra $A=\End_R(M)$ has finite global
dimension and is a maximal Cohen-Macauley $R$-module. Van den Bergh
informs us that if $R$ is a localization of a finitely generated
$k$-algebra, then the DG algebra $A$ is smooth and so the functor
$$M\stackrel{\bL}{\otimes}_R(-):D(R)\to D(A^{op})$$
is a categorical resolution of $D(R).$
\end{example}

\begin{remark} Note that in the last two examples the singular
varieties ($Y//G$ and $SpecR$ respectively) have rational
singularities [StVdB].
\end{remark}

\subsection{Resolution by Koszul duality} Let $A$ be an augmented DG algebra with
the augmentation ideal $A^+$. Consider the shifted complex $A^+[1]$
and the corresponding DG tensor coalgebra $BA:=T (A^+[1])$. The
differential in $BA$ depends on the differential in $A$ and the
multiplication in $A.$ It is called the bar construction of $A$. Its
graded linear dual $(BA)^*$ is again an augmented DG algebra called
the {\it Koszul dual} of $A$ and denoted $\check{A}$. The map
$\sigma :BA\to (B(A^{\op}))^{\op},$ $\sigma (b_1\otimes ...\otimes
b_n)=(-1)^{(\Sigma_{i<j}\bar{b}_i\bar{b}_j)+n}b_n\otimes ...\otimes
b_1$ is an isomorphism of DG coalgebras. (Here $\bar{b}$ is the
degree of $b$). Therefore the Koszul dual of $A^{\op}$ is
$(\check{A})^{\op}.$

Since $A$ is a DG algebra and $BA$ is a DG coalgebra the complex
$\Hom  (BA,A)$ is naturally a DG algebra. An element $\alpha \in
\Hom ^1(BA,A)$ is called a {\it twisting cochain} if it satisfies
the Maurer-Cartan equation $d\alpha +\alpha ^2=0.$ The projection of
$T A^+[1]$ onto its first component $A^+[1]$ followed by the
(shifted) identity map $A^+[1]\to A^+$ is the universal twisting
cochain which we denote by $\tau.$

Consider the tensor product $BA\otimes A$ with the differential
$d=d_{BA}\otimes 1+1\otimes d_A+t_{\tau}$ where $t_{\tau}(b\otimes
a)=b_{(1)}\otimes \tau(b_{(2)})a$ (here $b\mapsto b_{(1)}\otimes
b_{(2)}$ is the symbolic notation for the comultiplication map $BA
\to BA\otimes BA$). Then indeed $d^2=0$ and we denote the
corresponding complex by $BA\otimes _{\tau}A.$ It is
quasi-isomorphic to $k$ and is called the bar complex of $A.$ This
bar complex is naturally a right DG $A$-module. It is also a left DG
$BA$-comodule in the obvious way and hence a right DG
$\check{A}$-module. Therefore in particular $BA\otimes _{\tau}A$ is
a DG $A\otimes \check{A}$-module.

Similarly using $-\tau$ (which is a twisting cochain in the DG
algebra $\Hom ((BA)^{\op},A^{\op})^{\op})$) we define the
differential $d=d_A\otimes 1+1\otimes d_{BA}+s_{-\tau}$ on $A\otimes
BA,$ where $s_{-\tau}(a\otimes b)=-a\tau (b_{(1)})\otimes b{(2)}.$
Denote the resulting complex by $A\otimes _\tau BA;$ it is a left DG
$A$-module and a right DG $BA$-comodule in the obvious way. Hence in
particular $A\otimes _{\tau }BA$ is a DG $A^{\op}\otimes
\check{A}^{\op}$-module. It is again quasi-isomorphic to $k.$

Define the Koszul functor
$$K_A(-):=(-)\stackrel{\bL}{\otimes }_A(A\otimes _\tau BA):D(A)\to
D(\check{A}^{\op}).$$
 This functor is often full and faithful on the
subcategory $\Perf (A)$. Hence it defines a categorical resolution
of $D(A)$ in case the DG algebra $\check{A}^{\op}$ is smooth.  The
following lemma is proved in [ELOII].

\begin{lemma} Assume that an augmented DG algebra $A$ satisfies the
following properties.

i) $A^{<0}=0;$

ii) $A^0=k;$

iii) $\dim A^i<\infty $ for every $i.$ Then the Kozsul functor $K_A$
is full and faithful on the subcategory $\Perf (A).$
\end{lemma}

Here we consider another example.

\begin{prop} Let $A$ be an augmented finite dimensional  DG
algebra concentrated in nonpositive degrees. Assume in addition that
the augmentation ideal $A^+$ is nilpotent. Then the Koszul functor
$K_A:D(A)\to D(\check{A}^{\op})$ is a categorical resolution.
\end{prop}

 The proposition is equivalent to the following two lemmas.

\begin{lemma} Let $A$ be as in Proposition 5.6. Then the DG algebras
$\check{A}$ and $\check{A}^{\op}$ are smooth.
\end{lemma}

\begin{proof} It suffices to prove that the DG algebra $\check{A}$
is smooth. Indeed, replace $A$ by $A^{\op}.$

Let us combine the two versions of the bar complex in one. Consider
the tensor product $BA\otimes A\otimes BA$ with the differential
$$d=d_{BA}\otimes 1\otimes 1+1\otimes d_A\otimes 1+1\otimes 1\otimes
d_{BA}+t_{\tau}\otimes 1+1\otimes s_{-\tau}.$$ Then $d^2=0$ and
$BA\otimes A\otimes BA$ is a DG $(BA)^{\op}\otimes BA$-comodule in
the obvious way. We denote it by $BA\otimes _{\tau}A\otimes
_{\tau}BA.$ The map $\nu :BA\to BA\otimes A\otimes BA,$ $\nu
(b)=b_{(1)}\otimes 1\otimes b_{(2)}$ is a morphism of DG
$(BA)^{\op}\otimes BA$-comodules. Our assumption on $A$ implies that
$BA\otimes A\otimes BA$ is finite dimensional in each degree. Hence
its graded dual is $\check{A}\otimes A^*\otimes \check{A}.$ It is a
DG $\check{A}^{\op}\otimes \check{A}$-module which we denote by
$\check{A}\otimes _{\tau ^*}A^*\otimes _{\tau ^*}\check{A}.$

The dual of the morphism $\nu $ is the morphism of  DG
$\check{A}^{\op}\otimes \check{A}$-modules
$$\nu ^* :\check{A}\otimes _{\tau ^*}A^*\otimes _{\tau
^*}\check{A}\to \check{A},$$ where $\check{A}$ is the diagonal
 DG
$\check{A}^{\op}\otimes \check{A}$-module.

Notice that $\nu ^*$ is a quasi-isomorphism. Indeed, it suffices to
show that $\nu$ is such. Let $\epsilon :A\to k$ and $\eta :BA\to k$
be the augmentation and the counit respectively. Then the map $\eta
\otimes \epsilon :BA\otimes _{\tau}A\to k$ is a quasi-isomorphism.
Thus the morphism of complexes
$$\eta \otimes \epsilon \otimes 1:BA\otimes _{\tau} A\otimes
_{\tau}BA\to k\otimes BA=BA$$ is a quasi-isomorphism. But the
composition $\eta \otimes \epsilon \otimes 1\cdot \nu :BA\to BA$ is
the identity. Hence $\nu $ is a quasi-isomorphism.

We claim that $\check{A}\otimes _{\tau ^*}A^*\otimes _{\tau
^*}\check{A}$ is a perfect DG $\check{A}^{\op}\otimes
\check{A}$-module. Indeed consider the finite filtration of $A$ by
powers of the augmentation ideal and refine this filtration by the
image of the differential. (Note that $\cap _n(A^+)^n=0$ since $A^+$
is nilpotent.) This induces a filtration of the DG
$(BA)^{\op}\otimes BA$-comodule $BA\otimes _{\tau}A\otimes
_{\tau}BA$ with the subquotients being isomorphic to a direct sum of
shifted copies of $(BA)^{\op}\otimes BA$. This implies that the
subquotient of the dual filtration of $\check{A}\otimes _{\tau
^*}A^*\otimes _{\tau ^*}\check{A}$ are finite sums of free shifted
DG $\check{A}^{\op}\otimes \check{A}$-modules. That is
$\check{A}\otimes _{\tau ^*}A^*\otimes _{\tau ^*}\check{A}$ is a
perfect DG $\check{A}^{\op}\otimes \check{A}$-module. This proves
the lemma.
\end{proof}

\begin{lemma} Let $A$ be as in Proposition 5.6. Then the Kozsul functor
$K_A$ is full and faithful on the subcategory $\Perf (A).$
\end{lemma}

\begin{proof} Notice that $K_A(A)=k,$ hence it suffices to prove
that the natural map $A\to \bR \Hom _{\check{A}^{\op}}(k,k)$ is a
quasi-isomorphism.

As in the proof of Lemma 5.7 consider the filtration of $A$ by the
powers of the augmentation ideal $A^+$ refined by the image of the
differential. Then the induced filtration of the DG $BA$-comodule
$A\otimes _{\tau}BA$ has subquotients which are finite sums of
shifted copies of $BA.$ Notice that the DG $\check{A}^{\op}$-module
$BA$ is h-injective. (Indeed, $BA=(\check{A})^*$ since $BA$ is
finite dimensional in each degree.) Hence the DG
$\check{A}^{\op}$-module $A\otimes _{\tau}BA$ is h-injective so that
$$\bR \Hom _{\check{A}^{\op}}(k,k)=\Hom _{\check{A}^{\op}}(k,A\otimes
_{\tau}BA).$$ But $\Hom _{\check{A}^{\op}}(k,A\otimes _{\tau}BA)=A.$
This proves the lemma and finishes the proof of Proposition 5.6
\end{proof}

Here are some examples illustrating Proposition 5.6.

\begin{example} Let $V$ be a finite dimensional (graded) vector space
concentrated in degree zero. Consider the DG algebra $A=T
V/V^{\otimes 2}$ - the truncated tensor algebra on $V$. This DG
algebra is not smooth if $\dim V>0$. The Koszul dual DG algebra
$\check{A}$ has zero differential and is isomorphic to the tensor
algebra $T (V^*[-1])$, where $V^*[-1]$ is the dual space to $V$
placed in degree 1. This is a smooth DG algebra and the Koszul
functor $K_A$ is a categorical resolution of $D(A)$.
\end{example}

\begin{example} Let $A$ be a finite dimensional augmented algebra
(concentrated in degree zero) with the nilpotent augmentation ideal.
For example we can take the group algebra k[G] of a finite p-group G
in case the field $k$ is algebraically closed and has characteristic
p. Then again the Koszul functor $K_A$ is a categorical resolution
of $D(A).$
\end{example}

\section{Categorical resolution for schemes}

The following theorem was proved in [Rou].

\begin{theo} Let $X$ be a separated scheme of finite type over a
perfect field. Then there exists $E\in D^b(cohX)$ and $d\in \bbN$
such that $D^b(cohX)=\langle E\rangle _d.$
\end{theo}

Denote $A=\bR \Hom (E,E).$ The theorem implies that the functor
$$\bR \Hom (E, -):D(X)\to D(A)$$
induces an equivalence of subcategories $D^b(cohX)\simeq \Perf (A).$
Consequently $\Perf (A)=\langle A\rangle _d,$ i.e. $A$ is a strong
generator for $\Perf (A).$

\begin{remark} Unlike in [Rou] we do not regard the equivalence
$D^b(cohX)\simeq \Perf (A)$ with $A$ weakly smooth (or even smooth)
as saying that "going to the DG world, $X$ becomes regular". Indeed,
according to our definition only the "big" category $D(X)$ can be
smooth or not.
\end{remark}

We are going to strengthen Rouquier's result.

\begin{theo} Let $X$ be a separated scheme of finite type over a
perfect field $k.$  Then

a) There exists a classical generator $E\in D^b(cohX)$, such that
the DG algebra $A=\bR \Hom (E,E)$ is smooth and hence the functor
$$\bR Hom (E,-):D(X)\to D(A)$$
is a categorical resolution.

b) Given any other classical generator $E^\prime \in D^b(cohX)$ with
$A^\prime =\bR \Hom (E^\prime ,E^\prime)$, the DG algebras $A$ and
$A^\prime $ are derived equivalent (hence $A^\prime$ is also smooth)
and the categorical resolutions $D(A)$ and $D(A^\prime)$ of $D(X)$
are equivalent.
\end{theo}

\begin{proof} Let us first prove b) assuming a):

The functors $\bR \Hom (E,-),$ $\bR \Hom (E^\prime ,-)$ induce
respective equivalences $D^b(cohX)\simeq \Perf (A),$ $D^b(coh
X)\simeq \Perf (A^\prime).$ Consider the DG $A^\prime\otimes
A^{\op}$-module $\bR \Hom (E^\prime ,E)$ and the obvious morphism of
functors from $D^b(cohX)$ to $\Perf (A^\prime)$
$$\mu :\bR \Hom (E,-)\stackrel{\bL}{\otimes }_A\bR \Hom (E^\prime
,E)\to \bR \Hom (E^\prime ,-).$$ Then $\mu (E)$ is an isomorphism,
hence $\mu$ is an isomorphism. This implies that the functor
$$(-)\stackrel{\bL}{\otimes }_A\bR \Hom (E^\prime ,E):D(A)\to D(A^\prime)$$ induces an equivalence $\Perf
(A)\stackrel{\sim}{\to}\Perf (A^\prime).$ Thus it is an equivalence
by Lemma 2.12, so that $A$ and $A^\prime$ are derived equivalent and
the categorical resolutions $D(A)$ and $D(A^\prime)$ of $D(X)$ are
equivalent (Definition 4.3).

The proof of part a) requires some preparation. All schemes are
assumed to be $k$-schemes.

For a scheme of finite type $Z$ we denote by $Z^{\red}$ (resp.
$Z^{\ns}$, resp. $Z^{\sg}$) the scheme $Z$ with the reduced
structure (resp. the open subscheme of regular points, resp. the
closed subscheme of singular points).

\begin{defi} Let $Y$ be a scheme of finite type. An {\rm admissible covering} of
$Y$ is a finite collection of closed reduced subschemes $\{ Z_j\}$
such that the following set theoretical conditions hold

a) $Y=\cup Z_j,$

b) for every $j$ $$Z_j^{\sg} \subset \bigcup _{\{ s\vert Z_s\subset
Z_j\}}Z_s^{\ns}.$$
\end{defi}

\begin{example} For each scheme of finite type $Y$ there exists a  canonical admissible
covering: $Z_1=Y^{\red},$ $Z_{j+1}=(Z_j^{\sg})^{red}.$
\end{example}

\begin{defi} Let $Z$ be a reduced scheme of finite type.
We call $F\in D^b(cohZ)$ a
 {\rm quasi-generator} for $D(Z)$ if $F\vert _{Z^{\ns}}$
 is a compact generator
 for $D(Z^{\ns}).$
 \end{defi}

 For example if $Z$ is a reduced separated scheme of finite type and
$F\in \Perf (Z)$ is a generator for $D(Z)$ (Theorem 2.10b)), then it
is a quasi-generator.

\begin{defi} A {\rm generating data} on a  scheme of finite type
$Y$ is a collection $\{ Z_j,E_j\},$ where $\{Z_j\}$ is an admissible
covering of $Y$ and
 $E_j\in D^b(cohZ_j)$ is a quasi-generator for $D(Z_j)$ for each
 $j.$
\end{defi}

If $Y$ is a separated scheme of finite type, then it admits a
generating data. Indeed, we can take the canonical admissible
covering $\{Z_j\}$ as in Example 6.5 above, with $E_j\in \Perf
(Z_j)$ being a compact generator for $D(Z_j).$

\begin{prop} Let $Y$ be a separated scheme of finite type
with a generating data $\{Z_j,E_j\}.$ Let $i_j:Z_j\to Y$ be the
corresponding closed embedding. Then $$E:=\bigoplus _ji_{j*}E_j$$ is
a classical generator for $D^b(cohX).$
\end{prop}

\begin{proof} For a noetherian scheme $S$ and a closed subset $W\subset S$
we denote as usual by $D^b_W(cohS)$ the full subcategory of
$D^b(cohS)$ consisting of complexes whose cohomology sheaves are
supported on $W.$

We may assume that $ Z_i\subset Z_j$  implies that $i<j.$ Define the
closed subsets $W_j:=\cup _{s\leq j}Z_s.$ It suffices to prove for
each $j$ the following assertion
\medskip

$(*_j):$ $\quad$ The object $\bigoplus _{s\leq j}i_{s*}E_s$ is a
classical generator for the category $D_{W_j}^b(cohY).$
\medskip

Let us prove these assertions $(*_j)$ by induction on $j.$

$\underline{j=1}.$ We have $Z_1^{\ns}=Z_1,$ hence $E_1$ is a
classical generator for $D^b(cohZ_1)=\Perf (Z_1)=D(Z_1)^c$ (Theorem
2.2 b), Theorem 2.10a)).

\begin{lemma} Let $T$ be a separated noetherian scheme and $i:Z\to T$ be the embedding of
a reduced closed subscheme. Let $F\in D^b(cohZ)$ be a classical
generator. Then $i_*F$ is a classical generator for the category
$D^b_Z(cohT).$
\end{lemma}

\begin{proof} This follows from Lemmas 7.37, 7.41 in [Rou].
\end{proof}

Thus $i_{1*}E_1$ is a classical generator of
$D_{Z_1}^b(cohY)=D_{W_1}^b(cohY).$

$\underline{j-1\Rightarrow j}.$ Consider the following localization
sequence of triangulated categories
$$D_{W_{j-1}}^b(cohY)\to D_{W_j}^b(cohY)\to
D_{W_j-W_{j-1}}^b(coh(Y-W_{j-1})).$$

By our assumption $W_j-W_{j-1}\subset Z_j^{\ns}$ and $E_j\vert
_{Z_j^{\ns}}$  is a compact generator for $D(Z_j^{\ns})$, hence a
 classical generator for $D^b(cohZ_j^{\ns})=\Perf (Z_j)$. Since $W_j-W_{j-1}$ is an open subset
 of the scheme $Z_j^{\ns},$ we may consider it with the induced (reduced) scheme structure.
 Then
$E_j\vert _{W_j-W_{j-1}}$ is a classical generator for
$D^b(coh(W_j-W_{j-1}))=\Perf (W_j-W_{j-1}).$  So by Lemma 6.9
$(i_{j*}E_j)\vert _{Y-W_{j-1}}$ is a classical generator for
$D^b_{W_j-W_{j-1}}(coh(Y-W_{j-1})).$ Now the next Lemma 6.10 and the
induction hypothesis imply that
$$D^b_{W_j}(cohY)=\langle \bigoplus _{s\leq j}i_{s*}E_j\rangle,$$
which completes the induction step and proves the proposition.
\end{proof}

\begin{lemma} Let $\cS \to \cT \stackrel{\pi}{\to} \cT/\cS$ be a localization
sequence of triangulated categories. Let $G_1\subset \cS$ and
$G_2\subset \cT$ be subsets of objects such that $\cS =\langle
G_1\rangle$ and $\cT /\cS =\langle \pi (G_2)\rangle.$ Then $\cT
=\langle G_1\cup G_2\rangle .$
\end{lemma}

\begin{proof} Denote $\cT ^\prime :=\langle G_1\cup G_2\rangle \subset
T.$ Then $\cT ^\prime$ is by definition closed under direct
summands. It suffices to prove that $\cT /\cT ^\prime =0.$ But $\cS
\subset T^\prime \subset T.$ Hence $\cT/\cT ^\prime \simeq (\cT
/\cS)/(\cT ^\prime /\cS),$ and $\cT /\cS =\langle \pi (G_2)\rangle
\subset \cT ^\prime /\cS.$ Thus $\cT /\cT ^\prime =0.$
\end{proof}

In Proposition 6.8 above we have constructed a special classical
generator $E$ for the category $D^b(cohY).$ We will show that the DG
algebra $\bR \Hom (E,E)$ is smooth (if $k$ is perfect). This will
complete the proof of Theorem 6.3.

For a scheme of finite type $Y$ denote by $D_Y\in D^b(cohY)$ a
dualizing complex on $Y$ (which exists and is unique up to a shift
and a twist by a line bundle on each connected component of $Y$,
[Ha2],VI,Thm.3.1,\S 10 ), so that the functor
$$D(-):=\bR \Hom (-,D_Y):D^b(cohY)\to D^b(cohY)$$
is an anti-involution.
 Clearly, if $E$ is a classical generator
for $D^b(cohY),$ then so is $D(E).$ Recall that the duality commutes
with direct image functors under proper morphisms. In particular, if
$i:Z\to Y$ is a closed embedding and $F\in D^b(cohZ)$, then
$$i_*D(F)\simeq D(i_*F).$$
(Here one should take $D_Z=i^!D_Y.$, [Ha2],III,Thm.6.7;V,Prop.2.4.)

\begin{lemma} Let $\{Z_j,E_j\}$ be a generating data on a scheme of
finite type $Y.$ Then so is $\{Z_j, D(E_j)\}.$
\end{lemma}

\begin{proof} Fix $Z_j.$ We need to show that $D(E_j)\vert
_{Z_j^{ns}}$ is a compact generator of $D(Z_j^{\ns}).$ We have
 $D(E_j)\vert _{Z_j^{ns}}=D(E_j\vert _{Z_j^{\ns}}),$ hence the
 assertion follows from the next lemma.
 \end{proof}

 \begin{lemma} Assume that $W$ is a smooth scheme of finite type and
 $F\in \Perf
 (W)$ is a compact generator for $D(W).$ Then so is $D(F).$
 \end{lemma}

 \begin{proof} Since $W$ is regular, $\cO_W$ is a dualizing complex
 on $W.$ The functor $\bR \Hom (-,\cO _W)$ induces an
 anti-involution of the subcategory $\Perf (W).$ The lemma follows.
 \end{proof}

\begin{defi} Let $Y$ be a separated scheme of finite type with a
generating data $\{ Z_j, E_j\}.$  We call $\{Z_j, D(E_j)\}$ the {\it
dual generating data}.  We have $\oplus i_{j*}D(E_j)=D(\oplus
i_{j*}E_j),$ hence the dual generating data produces the dual
generator of $D^b(cohY).$
\end{defi}

\begin{prop} Assume that the field $k$ is perfect. Let $S,Y$ be
separated schemes of finite type. Let $\{ Z_j,E_j\}$ (resp. $\{ W_s,
F_s\}$) be a generating data on $S$ (resp. on $Y$). Then $\{
Z_j\times W_s, E_j\boxtimes F_s\}$ is a generating data for $S\times
Y.$
\end{prop}

\begin{proof} We need a lemma.

\begin{lemma}  Let $k$ be a perfect field, $A,B$ -
noetherian $k$-algebras. Assume that $A$ and $B$ are reduced. Then
so is $A\otimes B.$
\end{lemma}

\begin{proof} Let $p_1,...p_n\subset A$ (resp.
$q_1,...,q_m\subset B$) be the minimal primes. Then by our
assumption $A\subset \prod A/p_i,$ $B\subset \prod B/q_j.$ Hence
also $A\otimes B\subset \prod (A/p_i\otimes B/q_j).$ Therefore we
may assume that $A$ and $B$ are integral domains.

The algebra $A$ is the union of its finitely generated
$k$-subalgebras $A=\cup A_i,$ and $A\otimes B=\cup (A_i\otimes B).$
So we may assume that $A$ is finitely generated. Also, replacing $B$
by its fraction field, we may assume that $B$ is a field. Then by
Exercise II, 3.14 in [Ha1] it suffices to prove that the algebra
$A\otimes \overline{k}$ is reduced. But this algebra is the union of
its subalgebras which are etale over $A$ (since the field $k$ is
perfect). Therefore it is reduced. This proves the lemma.
\end{proof}

The lemma implies that for each $j,s$ the scheme $Z_j\times W_s$ is
a closed reduced subscheme of $S\times Y.$ Clearly
$$S\times Y=\bigcup _{j,s}Z_j\times W_s.$$
 By Proposition 3.8a) for each $j,s$
 $Z_j^{\ns}\times W_s^{\ns}\subset (Z_j\times W_s)^{\ns}.$
Actually the two schemes are equal. Indeed, let $x\in Z_j$ be a
point and $B$ the corresponding local ring. Let $y\in Z_j\times W_s$
be a nonsingular point lying over $x$ with the corresponding local
ring $C.$ Then $C$ is a flat over $B.$ Hence by [Gro],Prop.17.3.3 or
by [Ma],Thm.23.7i) $B$ is also regular.

Therefore
$$(Z_j\times W_s)^{\sg}=(Z_j^{\sg}\times W_s)\cup (Z_j\times
W_s^{\sg}).$$
 This implies that $\{Z_j\times W_s\}$
is an admissible covering of $X\times Y.$

We have
$$(E_j\boxtimes F_s)\vert _{(Z_j\times W_s)^{\ns}}=
(E_j\boxtimes F_s)\vert _{Z_j^{\ns}\times W_s^{\ns}}= (E_j\vert
_{Z_j^{\ns}})\boxtimes (F_s\vert _{W_s^{\ns}}).$$ Since $E_j\vert_{
Z_j^{\ns}}$ and $F_s\vert _{W_s^{\ns}}$ are compact generators of
$D(Z_j^{\ns})$ and $D(W_s^{\ns})$ respectively, then $(E_j\boxtimes
F_s)\vert _{(Z_j\times W_s)^{\ns}}$ is a compact generator by Lemma
3.14. This proves the proposition.
\end{proof}

\begin{cor} Let $\{Z_j,E_j\}$ be a generating data on a separated
scheme of finite type $X.$ Let $i_j:Z_j\to X$ denote the
corresponding closed embedding. Then $\{ Z_j\times Z_s, E_j\boxtimes
D(E_s)\}$ is a generating data on $X\times X.$ In particular, if
$E=\oplus _ji_{j*}E_j,$ then $E\boxtimes D(E)$ is a classical
generator for $D^b(coh(X\times X)).$
\end{cor}

\begin{proof} Follows from Lemma 6.11 and Proposition 6.14.
\end{proof}

\begin{prop}Let $Y$ be a separated scheme of finite type over a perfect field $k.$ Choose
a classical generator $E$ of $D^b(cohY)$ as in Proposition 6.8 above
and denote $A=\bR \Hom (E,E).$ Let $D(E)$ be the dual generator.
Then there exist canonical quasi-isomorphisms of DG algebras

a) $\bR \Hom (D(E),D(E))\simeq A^{\op},$

b) $\bR \Hom (D(E)\boxtimes E, D(E)\boxtimes E)\simeq A^{\op}\otimes
A.$

Let $\Delta :Y\to Y\times Y$ be the diagonal closed embedding.

c) There exists a canonical equivalence of categories
$D^b(coh(Y\times Y))\simeq \Perf (A^{\op}\otimes A )$ which takes
the object $\Delta _*(D_Y)$ to the diagonal DG $A ^{\op}\otimes A
$-module $A.$ In particular the DG algebra $A$ is smooth.
\end{prop}

We prove this proposition in Subsection 6.1 below.

Part a) of Theorem 6.3 now follows. Indeed, let $E$ be a classical
generator for $D^b(cohX)$ as in Proposition 6.8, then by Proposition
6.17 the DG algebra $A=\bR \Hom (E,E)$ is smooth.
\end{proof}

\subsection{Proof of Proposition 6.17}

a). Since $D:D^b(cohY)\to D^b(cohY)$ is an anti-involution the map
$$D:\Ext (E,E)\to \Ext  (D(E),D(E))$$
is an isomorphism. Choose h-injective resolutions $E\to I,$ $D_Y\to
J,$ so that $A=\Hom (I,I)$ and $D(E)=\cH om (I,J).$ Let $\rho :\cH
om (I,J)\to K$ be an h-injective resolution, so that $B:=\Hom
(K,K)=\bR \Hom (D(E),D(E)).$ We have the natural homomorphism of DG
algebras
$$\epsilon :A^{\op}\to \Hom (\cH om (I,J),\cH om (I,J))$$
such that the composition of $\epsilon$ with the map
$$\Hom (\cH om(I,J),\cH om(I,J))\stackrel{\rho _*}{\to}\Hom (\cH
om(I,J),K)$$ is a quasi-isomorphism (since this composition induces
the map $D$ above between the $\Ext$-groups). Notice also that the
map $\rho ^*:B\to \Hom (\cH om (I,J),K)$ is a quasi-isomorphism. It
follows from Lemma 2.14 that the DG algebra
$$\left(
    \begin{array}{cc}
      B & \Hom (\cH om (I,J)[1],K) \\
      0 & A^{\op} \\
    \end{array}
  \right)
$$
(where the differential is defined using the above maps) is
quasi-isomorphic to DG algebras $B$ and $A^{\op}$ by the obvious
projections. This proves a).

b). The proof is similar and we will use the same notation. In
addition to resolutions $E\to I,$ $D(E)\to K$ choose an h-injective
resolution $\sigma :D(E)\boxtimes E\to L,$ so that $\bR \Hom
(D(E)\boxtimes E,D(E)\boxtimes E)=\Hom (L,L).$ We need a couple of
lemmas.

\begin{lemma} The obvious morphism of sheaves of DG algebras on
$Y\times Y$
$$\cH om (K,K)\boxtimes \cH om (I,I)\to \cH om (K\boxtimes
I,K\boxtimes I)$$ is a quasi-isomorphism.
\end{lemma}

\begin{proof} The question is local so we may assume that $Y=SpecB$
for some noetherian $k$-algebra B. Then we can find bounded above
complexes $P,Q$ of free $B$-modules of finite rank which are
quasi-isomorphic to $D(E)$ and $E$ respectively. Similarly, we can
find bounded below complexes $M,N$ of injective $B$-modules which
are quasi-isomorphic to $D(E)$ and $E$ respectively. It suffices to
prove that the corresponding map
$$\Hom _B(P,M)\otimes \Hom _B(Q,N)\to \Hom _{B\otimes B}(P\otimes
Q,M\otimes N)$$ is an isomorphism. This follows from the formula
$$\Hom _B(B,S)\otimes \Hom _B(B,T)=S\otimes T=\Hom _{B\otimes
B}(B\otimes B,S\otimes T)$$ for any $B$-modules $S,T.$
\end{proof}

\begin{lemma} $\bR \Gamma (\cH om (I,I))=\Gamma (\cH om (I,I)),$
$\bR \Gamma (\cH om (K,K)) =\Gamma (\cH om (K,K)).$
\end{lemma}

\begin{proof} It suffices to prove the first assertion.
Since $I$ is quasi-isomorphic to a bounded complex we can find a
quasi-isomorphism $\theta :I\to I^\prime,$ where $I^\prime$ is a
bounded below complex of injective quasi-coherent sheaves which are
also injective in the category $\Mod _{\cO _Y}$ of {\it all} $\cO
_Y$-modules [Ha],II,Thm.7.18. Both $I$ and $I^\prime$ are
h-injective in $D(Y),$ so the map $\theta$ is a homotopy
equivalence. Hence also $\theta _*:\cH om (I,I)\to \cH om
(I,I^\prime)$ is a homotopy equivalence. So it suffices to prove
that $\bR \Gamma (\cH om (I,I^\prime ))=\Gamma (\cH om (I,I^\prime
)).$ The complex $I^\prime $ is h-injective in the category $C(\Mod
_{\cO _Y}),$ hence $\cH om (I,I^\prime)$ is {\it weakly injective}
in this category in the terminology of [Sp],Prop.5.14. Hence $\bR
\Gamma (\cH om (I,I^\prime ))=\Gamma (\cH om (I,I^\prime ))$ by
Proposition 6.7 in [Sp].
\end{proof}

Recall the Kunneth formula [Lip],Th.3.10.3: the natural map
$$\bR \Gamma (S)\otimes \bR \Gamma (T)\to \bR \Gamma (S\boxtimes
T)$$ is a quasi-isomorphism for all $S,T\in D(Y).$ Applying this to
$S=\cH om (K,K),\ T=\cH om (I,I)$ and using Lemmas 6.18 and 6.19 we
conclude that the composition of the homomorphism of DG algebras
$B\otimes A \to \Hom (K\boxtimes I,K\boxtimes I)$ with the map
$\sigma _*:\Hom (K\boxtimes I,K\boxtimes I)\to \Hom (K\boxtimes
I,L)$ is a quasi-isomorphism. Now as in the proof of part a) we
conclude that the DG algebra
$$\left(
    \begin{array}{cc}
      \Hom (L,L) & \Hom (\cH om (K\boxtimes I)[1],L) \\
      0 & B\otimes A \\
    \end{array}
  \right)
$$
is  quasi-isomorphic to both $\Hom (L,L)$ and $B\otimes A.$ But
$B\simeq A^{\op}$ by a), which proves b).

c). We still use the same notation. By definition $I$ is a DG
$A^{\op}$-module (more precisely, a sheaf of DG $A^{\op}$-modules),
hence $\cH om (I,J)$ is a DG $A$-module via the action on $I.$ It
follows that
$$\Psi (-):=\bR \Hom (\cH om (I,J)\boxtimes I,-)$$
is a functor from $D(Y\times Y)$ to $D(A^{\op}\otimes A).$ We claim
that $\Psi$ induces an equivalence between $D^b(coh(Y\times Y))$ and
$\Perf (A^{\op}\otimes A).$ Indeed, by Corollary 6.16 $L$ is a
classical generator for $D^b(coh(Y\times Y))$. Hence it suffices to
show that $\Psi (L)=A^{\op}\otimes A.$ Consider the commutative
diagram
$$\begin{array}{ccccc}
          B\otimes A & \to  & \Hom (K\boxtimes I,K\boxtimes I) & \stackrel{\sigma _*}{\to} & \Hom (K\boxtimes I,L) \\
          \downarrow & & \downarrow & & \downarrow \\
          \Hom (\cH om (I,J),K)\otimes A & \to  & \Hom (\cH om (I,J)\boxtimes I,K\boxtimes I) & \to  & \Hom (\cH om (I,J)\boxtimes I,L)
        \end{array}
$$
where the maps in the top row were considered in the proof of b)
(and the composition is a quasi-isomorphism), and the vertical
arrows are induced by the quasi-isomorphism $\cH om (I,J)\to K.$ At
least the left and right vertical arrows are quasi-isomorphisms.
Thus the composition of arrows in the bottom row (which are maps of
DG $A^{\op}\otimes A$-modules) is a quasi-isomorphism. Now recall
the quasi-isomorphism of DG $A^{\op}$-modules $A^{\op}\to \Hom (\cH
om (I,J),K)$ from the proof of a). As a result we obtain a
quasi-isomorphism of DG $A^{\op}\otimes A$-modules
$$A^{\op}\otimes A\to \Hom (\cH om (I,J),K)\otimes A\to \Hom (\cH om
(I,J)\boxtimes I,L)=\Psi (L)$$ as required.

Now it is easy to see that $\Psi (\Delta _*D_Y)=A$ (with the
diagonal DG $A^{\op}\otimes A$-module structure). Namely, denote by
$Y\stackrel{p}{\leftarrow} Y\times Y \stackrel{q}{\to} Y$ the two
projections. Then
$$\begin{array}{rcl}
  \Psi (\Delta _*D_Y) & = & \bR \Hom (\cH om (I,J)\boxtimes I,\Delta _*D_Y) \\
   & = & \bR \Hom (p^*I, \bR \cH om (q^* \cH om (I,J),\Delta _*D_Y)) \\
   & = &  \bR \Hom (p^*I, \Delta _*\cH om (\bL \Delta ^* q^* \cH om (I,J),J))\\
   & = & \bR \Hom (p^*I, \Delta _*\cH om (\cH om (I,J),J)) \\
   & = & \bR \Hom (\bL \Delta ^* p^*I, \cH om ( \cH om (I,J),J)) \\
   & = & \bR \Hom (I, \cH om ( \cH om (I,J),J))
\end{array}
$$
Note that all these equalities are quasi-isomorphisms of DG
$A^{\op}\otimes A$-modules. Note also that the natural map $I\to \cH
om (\cH om (I,J),J)$ is a quasi-isomorphism of DG $A^{\op}$-modules.
Hence we obtain a quasi-isomorphism of DG $A^{\op}\otimes A$-modules
$$\bR \Hom (I, \cH om ( \cH om (I,J),J))=\Hom (I,I)=A$$
as required. This proves c) and the proposition.

The proof of Proposition 6.17 gives more than stated. Namely, using
similar arguments we obtain the following result.

\begin{prop} Let $Y,Z$ be noetherian $k$-schemes, $F_1,F_2\in
D^b(cohY),$ $G_1,G_2\in D^b(cohZ).$

a) There exists a natural quasi-isomorphism of complexes
$$\bR \Hom (F_1,F_2)\otimes \bR \Hom (G_1,G_2)\simeq \bR \Hom
(F_1\boxtimes G_1,F_2\boxtimes G_2).$$

b) There exists a natural quasi-isomorphism of DG algebras
$$\bR \Hom (F_1,F_1)\otimes \bR \Hom (G_1,G_1)\simeq \bR \Hom
(F_1\boxtimes G_1,F_1\boxtimes G_1).$$
\end{prop}

\subsection{Concluding remarks on Theorem 6.3} Assume that the field $k$ is perfect.
 By Theorem 6.3 for a separated scheme of finite type $X$
 there exists a canonical (up to equivalence)
categorical resolution of singularities $D(X)\to D(A).$ It has the
flavor of Kozsul duality (Subsection 5.1) and may be called the
"inner" resolution. It has two notable properties: 1) The DG algebra
$A$ is derived equivalent to $A^{\op}$ (indeed, we can use a
classical generator $E$ for $D^b(cohX)$ or its dual $D(E)$); 2) $A$
usually has unbounded cohomology. In the forthcoming paper [Lu2] we
suggest a different type of a categorical resolution of $D(X):$ the
resolving smooth DG algebra has bounded cohomology, but is usually
not derived equivalent to its opposite.

\subsection{Some remarks on duality for noetherian
schemes}

\begin{defi} Let $\cD$ be a triangulated category. An object $M\in
\cD$ is called {\rm homologically} (resp. {\rm cohomologically})
{\rm finite} if for every $N\in \cD,$ $\Hom (M,N[i])=0$ for $\vert
i\vert  >>0$ (resp. $\Hom (N,M[i])=0$ for $\vert i\vert  >>0.$)
Denote by $\cD _{\hf}$ (resp. $\cD _{\chf}$) the full triangulated
subcategory of $\cD$ consisting of homologically (resp.
cohomologically) finite objects.
\end{defi}

\begin{defi} For a noetherian scheme $Y$ consider the bifunctor
$$\bR \cH om (-,-):D^b(cohY)^{\op}\times D^b(cohY)\to D^+(cohY).$$
We say that $F\in D^b(cohY)$ is {\rm locally homologically} (resp.
{\rm locally cohomologically}) {\rm finite} if $\bR \cH om(F,G)\in
D^b(cohY)$ (resp. $\bR \cH om(G,F)\in D^b(cohY)$) for all $G\in
D^b(cohY).$ Let $D^b(cohY)_{\lhf}$ (resp. $D^b(cohY)_{\lchf}$) be
the full subcategory of $D^b(cohY)$ consisting of locally
homologically (resp. locally cohomologically) finite objects.
\end{defi}

Let $Y$ be a noetherian scheme with a dualizing complex $D_Y\in
D^b(cohY).$ The duality equivalence
$$D(-)=\bR \cH om(-,D_Y):D^b(cohY)^{\op}\stackrel{\sim}{\to} D^b(cohY)$$
induces equivalences
$$D:D^b(cohY)^{\op}_{\hf}\stackrel{\sim}{\to}D^b(cohY)_{\chf},$$
$$D:D^b(cohY)^{\op}_{\lhf}\stackrel{\sim}{\to}D^b(cohY)_{\lchf}.$$

Denote by $\Fid (Y)\subset D^b(cohY)$ the full subcategory
consisting of complexes which are quasi-isomorphic to a finite
complex of injectives in $QcohX.$

\begin{lemma} Let $Y$ be a noetherian scheme with a dualizing complex,
$F\in D^b(cohY).$ Then the conditions a),b),c) are equivalent

a) $F\in \Perf (Y),$

b) $F\in D^b(cohY)_{\lhf},$

c) $F\in D^b(cohY)_{\hf}.$

Also the dual conditions d),e),f) are equivalent

d) $F\in \Fid (Y),$

e) $F\in D^b(cohY)_{\lchf},$

f) $F\in D^b(cohY)_{\chf}.$
\end{lemma}

\begin{proof} It is obvious that $a)\Rightarrow b)\Rightarrow c).$

Assume that $F\in D^b(cohY)_{hf}.$ Let $U=SpecC$ be an open affine
subscheme of $Y.$ Then $C$ is a noetherian $k$-algebra. Choose a
bounded above complex $P=...\to P^n\stackrel{d^n}{\to}P^{n+1}\to
...$ of free $C$-modules of finite rank which is quasi-isomorphic to
$F\vert _U.$ Then for $n<<0$ the truncation
$$\tau _{\geq n}P=0\to \Ker d^n \to P^n\to P^{n+1}\to ...$$
is also quasi-isomorphic to $F\vert _U.$ Let $x\in U$ be a closed
point with the residue field $k(x).$ Since $\Ext ^m(F,k(x))=0$ for
$m>>0,$ this implies that  $\Ext ^{>0}_C(\Ker d^n,k(x))=0.$ Hence
the $C$-module $\Ker d^n$ is free at $x$ for $n>>0.$ Hence it is
free in an open neighborhood of $x.$ So $F\in \Perf (Y).$

Again the implications $d)\Rightarrow e)\Rightarrow f)$ are clear.
Actually $d)\Leftrightarrow e)$ by [Ha2],II,Prop.7.20. It remains to
prove that $f)\Rightarrow e).$ Let $F\in D^b(cohY)_{chf}.$ Then
$D(F)\in D^b(cohY)_{hf},$ so also $D(F)\in D^b(cohY)_{lhf}$ by
$c)\Rightarrow b).$ But then $D(D(F))=F\in F\in D^b(cohY)_{lchf}.$
\end{proof}

\begin{cor} In the above notation the duality functor induces an
equivalence $D:\Perf (Y)^{\op}\stackrel{\sim}{\to} \Fid (Y).$
\end{cor}

\begin{proof} This follows from Lemma 6.23.
\end{proof}

Recall that a noetherian scheme $Y$ is called Gorenstein, if all its
local rings are Gorenstein local rings. Then $Y$ is Gorenstein if
and only if $\cO _Y$ is a dualizing complex on $Y$ [Ha2].

\begin{lemma} A noetherian scheme $Y$
is Gorenstein if and only if $\Perf (Y)=\Fid (Y).$
\end{lemma}

\begin{proof} The functor $\bR \cH om (-,\cO _Y):D^b(cohY)^{\op}\to
D^+(cohY)$ induces an equivalence $\Perf (Y)^{\op}\to \Perf (Y).$ So
if $Y$ is Gorenstein then $\Perf (Y)=\Fid (Y)$ by Corollary 6.24.

Conversely if $\Perf (Y)=\Fid (Y)$ then in particular $\cO _Y\in
\Fid(Y).$ In any case $\bR \cH om(\cO_Y,\cO _Y)=\cO _Y,$ so $\cO _Y$
is a dualizing complex on $Y$ by [Ha2],Ch.V,Prop.2.1.
\end{proof}

\subsection{Canonical categorical resolution as a mirror which switches "perfect" and "bounded"}

Let the field $k$ be perfect and $X$ be a separated $k$-scheme  of
finite type with a dualizing complex $D_X\in D^b(cohX).$

Choose a classical generator $E\in D^b(cohX)$ and denote the
corresponding equivalence
$$\Psi (-):=\bR \Hom (E,-):D(cohX)\to \Perf (A),$$
where $A=\bR \Hom (E,E)$ (Theorem 6.3). Consider also the
equivalence
$$\Psi \cdot D(-)=\bR \Hom (E,\bR \cH om (-,D_X)):D(cohX)^{\op}\to
\Perf (A).$$

\begin{defi} A DG $A$-module $M$ is called {\rm bounded} if
$H^i(M)=0$ for $\vert i\vert >>0.$ Denote by $D^b(A)\subset D(A)$
the full subcategory consisting of bounded DG modules. Put $\Perf
(A)^b=\Perf (A)\cap D^b(A).$
\end{defi}

\begin{prop}
a) The functor $\Psi$ induces an equivalence $\Fid
(X)\stackrel{\sim}{\to} \Perf (A)^b;$

b) The composition $\Psi \cdot D$ induces an equivalence $\Perf
(X)^{\op}\simeq \Perf (A)^b.$
\end{prop}

\begin{proof} a). Clearly $\Psi (\Fid(X))\subset \Perf (A)^b.$ Vice
versa, assume that $\Psi (G)\in \Perf (A)^b$ for some $G\in
D^b(cohX).$ Since $E$ is a classical generator for $D^b(cohX)$ the
complex $\bR \Hom (F,G)$ has bounded cohomology for all $F\in
D^b(cohX).$ That is $G\in D^b(cohX)_{chf}.$ But then $F\in \Fid (X)$
by Lemma 6.23.

b). Follows from a) and Corollary 6.24.
\end{proof}

Recall the triangulated category of singularities
$D_{\sg}(X)=D^b(cohX)/\Perf (X)$ ([Or]).

\begin{cor} The functor $\Psi \cdot D$ induces an equivalence
$$D_{\sg}(X)^{\op}\simeq \Perf (A)/\Perf (A)^b.$$
\end{cor}

\begin{cor} Assume that $X$  Gorenstein. Then in the context of Proposition
6.27 the functor  $\Psi $ induces an equivalence $\Perf (X)\to \Perf
(A)^b.$ Hence in particular $D_{\sg}(X)\simeq \Perf (A)/\Perf
(A)^b.$
\end{cor}

\begin{proof} Since $X$ is Gorenstein $\Perf (X)=\Fid (X).$ Hence
the corollary follows from Proposition 6.27a).
\end{proof}

\subsection{Connection with the stable derived category of a locally
noetherian Grothendieck category} Let $\cA$ be a locally noetherian
Grothendieck category such that its derived category $D(\cA)$ is
compactly generated. Denote by $\noeth \cA \subset \cA$ the full
subcategory of noetherian objects. Let $\Inj \cA \subset \cA$ be the
full subcategory of injective objects and consider its homotopy
category $K(\Inj \cA):=Ho (\Inj \cA).$ Let $S(\cA) \subset K (\Inj
\cA)$ be the full triangulated category of acyclic complexes. In
[Kr] the following assertions were proved:

1) The natural diagram of triangulated categories and exact functors
$$S(\cA)\stackrel{I}{\lto} K(\Inj \cA)\stackrel{Q}{\lto} D(\cA)$$
is a localization sequence, in particular $D(\cA)\simeq K(\Inj
\cA)/S(\cA).$

2) The functors $I,Q$ have left adjoints $I_\lambda ,Q_\lambda$ and
right adjoints $I_\rho ,Q_\rho$ respectively.

3) The category $K(\Inj \cA)$ is cocomplete and compactly generated.

4) The functor $Q$ induces an equivalence of categories
$$K(\Inj \cA )^c\stackrel{\sim}{\lto} D^b(\noeth \cA)$$
with the quasi-inverse being induced by $Q_\rho.$

In [Kr] the category $S(\cA)$ is called the stable derived category
of $\cA$ and Krause suggests a deeper study of the category $K(\Inj
\cA).$

Let $k$ be a perfect field and $X$ be a separated $k$-scheme of
finite type. The Grothendieck category $\cA =Qcoh X$ is locally
noetherian with $\noeth \cA =coh X.$ The derived category
$D(X)=D(\cA)$ is compactly generated (Theorem 2.10). We denote $\Inj
X=\Inj \cA,$ $K(\Inj X)=K(\Inj \cA),$ $S(X)=S(\cA).$ So we obtain
the localization sequence
$$S(X)\stackrel{I}{\lto} K(\Inj X)\stackrel{Q}{\lto} D(X).$$

\begin{prop} The pair $(K(\Inj X),Q_{\rho})$ is a categorical
resolution of $D(X)$ which is equivalent to the canonical resolution
constructed in Theorem 6.3. In particular the category $K(\Inj X)$
is smooth.
\end{prop}

\begin{proof} Let $K_{\Inj }(X)\subset K(\Inj X)$ be the full
subcategory consisting of h-injective complexes (with injective
components). The functor $Q$ induces an equivalence $K_{\Inj
}(X)\stackrel{\sim}{\to }D(X).$ Its quasi-inverse (="taking
h-injective resolution") composed with the inclusion
$K_{\Inj}(X)\subset K(\Inj X)$ is the functor $Q_{\rho}.$ Thus we
may identify $D(X)$ with $K_{\Inj}(X).$ Hence by 4) above the
category $D^b(cohX)$ is identified with $K(\Inj X)^c.$

Let $E\in K_{\Inj}(X)$ be a classical generator of $D^b(cohX),$
hence also of $K(\Inj X)^c;$ put $A=\Hom (E,E).$ By Theorem 6.3 the
functor
$$\Psi _E:K_{\Inj}(X)\to D(A),\quad \Psi _E(-)=\Hom (E,-)$$
is a categorical resolution. So it suffices to prove that the
functor $\Psi ^\prime _E:K(\Inj X)\to D(A)$ defined by the same
formula is an equivalence.

We know that the category $K(\Inj X)$ is cocomplete. Hence by
Theorem 2.2 b) $E\in K(\Inj X)^c$ is a compact generator for $K(\Inj
X).$ Now one shows that $\Psi ^\prime _E$ is an equivalence by
copying the proof of Proposition 2.6.
\end{proof}

We thank Michel Van den Bergh for pointing to us the connection
between our categorical resolution of $D(X)$ and the category
$K(\Inj X).$

\medskip

\noindent{\bf Question.} For which locally noetherian Grothendieck
categories $\cA$ (such that $D(\cA)$ is compactly generated) the
category $K(\Inj \cA)$ is smooth (hence a categorical resolution of
$D(\cA)$)?

\end{document}